\documentclass[12pt, reqno]{amsart}
\usepackage{amsmath, amsthm, amscd, amsfonts, amssymb}
\usepackage{bm}
\usepackage[all,cmtip]{xy}
\usepackage{xcolor}
\newtheorem{theorem}{Theorem}[section]
\newtheorem{lemma}[theorem]{Lemma}
\newtheorem{proposition}[theorem]{Proposition}

\theoremstyle{definition}
\newtheorem{definition}[theorem]{Definition}

\newtheorem{remark}[theorem]{Remark}
\numberwithin{equation}{section}

\newcommand{\vp}{\varphi}

\newcommand{\clb}{\mathcal{B}}

\newcommand{\cle}{\mathcal{E}}

\newcommand{\clh}{\mathcal{H}}

\newcommand{\cll}{\mathcal{L}}
\newcommand{\clm}{\mathcal{M}}

\newcommand{\clq}{\mathcal{Q}}

\newcommand{\cls}{\mathcal{S}}

\newcommand{\clw}{\mathcal{W}}

\newcommand{\D}{\mathbb{D}}

\newcommand{\T}{\mathbb{T}}
\newcommand{\Z}{\mathbb{Z}}
\newcommand{\raro}{\rightarrow}

\begin{document}

\title[Invariant subspaces of Jordan blocks of $H^2(\mathbb{D}^n)$]{Invariant subspaces of Jordan blocks of $H^2(\mathbb{D}^n)$}


\author[Sneha B]{Sneha B}
\address{Statistics and Mathematics Unit, Indian Statistical Institute, 8th Mile, Mysore Road, Bangalore, 560059, India}
\email{sneharbkrishnan@gmail.com, rs\textunderscore math2105@isibang.ac.in}

\author[Sarkar]{Jaydeb Sarkar}
\address{Statistics and Mathematics Unit, Indian Statistical Institute, 8th Mile, Mysore Road, Bangalore, 560059, India}
\email{jay@isibang.ac.in, jaydeb@gmail.com}
		
\author[Seto]{Michio Seto}
\address{National Defense Academy of Japan, Yokosuka 239-8686, Japan}
\email{mseto@nda.ac.jp}
		
\subjclass[2020]{47A15, 30H10, 42B30, 32A35}
\keywords{Model spaces, polydisc, Hardy space, invariant subspaces, reducing subspaces, inner functions, unitary equivalence}

\begin{abstract}
It is known that invariant subspaces of classical Jordan blocks of the Hardy space over the open unit disc are described by factorizations of inner functions. In the polydisc setting, Jordan blocks are tensor products of one-variable Jordan blocks. We provide representations of their doubly commuting invariant subspaces.
\end{abstract}
	
\maketitle
	
\tableofcontents

\section{Introduction}

Product domains, such as the open unit polydisc $\D^n$ in $\mathbb{C}^n$, and tensor products of Hilbert spaces, such as the Hardy space over $\D^n$ denoted by $H^2(\D^n)$, are closely related. For instance, one can identify $H^2(\D^n)$ with the $n$-fold tensor products
\[
H^2(\D^n) \cong \underbrace{H^2(\D) \otimes \cdots \otimes H^2(\D)}_{n \text{ times}},
\]
via the unitary operator $U$ that maps $z^k := z_1^{k_1} \cdots z_n^{k_n}$ to $z^{k_1} \otimes \cdots \otimes z^{k_n}$ for all $k=(k_1, \ldots, k_n) \in \Z_+^n$. Recall that $H^2(\D^n)$ consists of holomorphic functions $f$ on $\D^n$ with
\[
\|f\|_{H^2(\D^n)}: = \bigg(\sup\limits_{0\le r<1}\int_{\mathbb{T}^n}|f(rz_1,\ldots, rz_n)|^2 d \mu \bigg)^{\frac{1}{2}}<\infty,
\]
where $d\mu$ is the normalized Lebesgue measure on $\mathbb{T}^n$ \cite{MR255841}. For each $i \in I_n$ (where $I_n = \{1, \ldots, n\}$), the multiplication operator $T_{z_i}$ on $H^2(\D^n)$, defined by
\[
T_{z_i} f = z_if,
\]
for $f \in H^2(\D^n)$, is an isometry. In view of the unitary $U$, we have the identification:
\[
T_{z_i} \cong I_{H^2(\D)} \otimes \cdots I_{H^2(\D)} \otimes \underbrace{T_z}_{i \text{ th}} \otimes I_{H^2(\D)} \cdots \otimes I_{H^2(\D)},
\]
for all $i \in I_n$. This shows that $(T_{z_1}, \ldots, T_{z_n})$ is an $n$-tuple of doubly commuting isometries. Recall that for $\{T_i\}_{i=1}^n \subseteq \clb(\clh)$ (where $\clb(\clh)$ denotes the space of bounded linear operators on $\clh$), the $n$-tuple $(T_1, \ldots, T_n)$ is said to be \textit{doubly commuting} if for all $i < j$,
\[
T_i T_j = T_j T_i \text{ and } T^*_i T_j = T_j T_i^*.
\]
In the context of $n$-tuples of doubly commuting operators, we always assume that $n>1$.

Now we turn to quotient modules of $H^2(\D^n)$. A closed subspace $\clq$ of $H^2(\D^n)$ is said to be a \textit{quotient module} or \textit{model space} if
\[
T_{z_i}^* \clq \subseteq \clq,
\]
for all $i \in I_n$. The structure of quotient modules is quite intricate, and these complications are known to be connected with various problems in the theory of Hilbert function spaces \cite{CG, DY}. In the one-variable case, however, quotient modules are concrete \cite{HB}: A nontrivial proper closed subspace $\clq$ of $H^2(\D)$ is a quotient module if and only if there exists a nonconstant inner function $\theta \in H^\infty(\D)$ such that $\clq = \clq_\theta$, where
\[
\clq_{\theta}:=H^2(\D) \ominus \theta H^2(\D).
\]
In this case, define the \textit{model operator} or \textit{Jordan block} $S_\theta \in \clb(\clq_\theta)$ \cite[page 36]{HB} by
\[
S_\theta = P_{\clq_{\theta}}T_z|_{\clq_{\theta}},
\]
where $P_{\clq_{\theta}}$ denotes the orthogonal projection of $H^2(\D^n)$ onto $\clq_\theta$. Occasionally, we also refer $\clq_\theta$ itself as a Jordan block. Recall that $H^\infty(\D^n)$ is the space of all bounded analytic functions on $\D^n$, and a function $\vp \in H^\infty(\D^n)$ is inner if $|\vp(z)| = 1$ a.e. on $\T^n$.

In the case $n>1$, the simplest examples of quotient modules are again given by $n$-fold tensor products of Jordan blocks, or more generally, quotient modules of $H^2(\D)$:
\[
\clq=\clq_1\otimes\cdots\otimes\clq_n,
\]
where $\{\clq_i\}_{i=1}^n$ are quotient modules of $H^2(\D)$. In fact, these are precisely the doubly commuting quotient modules of $H^2(\D^n)$; this was proved for $n=2$ in \cite{INS} and for general $n>1$ in \cite{MR3272037} (also see \cite{DanGuo}). Recall that a quotient module $\clq \subseteq H^2(\D^n)$ is said to be \textit{doubly commuting} if $(P_\clq T_{z_1}|_\clq, \ldots, P_\clq T_{z_n}|_\clq)$ on $\clq$ is doubly commuting.

A doubly commuting quotient module $\clq$ as above is called a Jordan block of $H^2(\D^n)$ if $\clq_i$ is a Jordan block of $H^2(\D)$ for all $i \in I_n$. Equivalently:

\begin{definition}
Jordan blocks of $H^2(\D^n)$ are quotient modules of the form
\[
\clq_\Theta = \clq_{\theta_1} \otimes \cdots \otimes \clq_{\theta_n},
\]
where $\{\theta_i\}_{i=1}^n$ are nonconstant inner functions in $H^\infty(\D)$.
\end{definition}

Now we turn to invariant subspaces (or submodules) of Jordan blocks of $H^2(\D^n)$. We begin with the classical case $n=1$ case: A closed subspace $\clw \subseteq \clq_\theta$ is called a \textit{submodule} (or, \textit{invariant subspace of $S_\theta$}) if $S_\theta \clw \subseteq \clw$. Every submodule $\clw$ of $\clq_\theta$ admits the representation
\[
\clw = \vp \clq_\psi,
\]
for some inner functions $\vp$ and $\psi$ such that $\theta = \vp \psi$ (see \cite[page 38]{HB} and also Section \ref{sec: subm clas JB}). Now, assume that $n>1$. Fix a Jordan block $\clq_\Theta$ of $H^2(\D^n)$. For $j \in I_n$, set
\[
S_{\Theta_j} = I_{\clq_{\theta_1}} \otimes \cdots I_{\clq_{\theta_{j-1}}} \otimes \underbrace{S_{\theta_j}}_{j-\text{th}} \otimes I_{\clq_{\theta_{j+1}}} \cdots \otimes I_{\clq_{\theta_n}}.
\]
Then
\[
S_\Theta:=(S_{\Theta_1}, \ldots, S_{\Theta_n}),
\]
defines an $n$-tuple of doubly commuting contractions on $\clq_\Theta$. A \textit{submodule of the Jordan block $\clq_\Theta$} is a closed subspace $\clm \subseteq \clq_\Theta$ with
\[
S_{\Theta_i} \clm \subseteq \clm,
\]
for all $i \in I_n$. Such a submodule $\clm$ induces a commuting $n$-tuple $R_{\Theta, \clm} = (R_1, \ldots, R_n)$ on $\clm$, where
\[
R_i = S_{\Theta_i}|_\clm \in \clb(\clm),
\]
for all $i \in I_n$. We call $\clm$ \textit{doubly commuting} if $R_{\Theta, \clm}$ on $\clm$ is doubly commuting. The following is one of the main results of this paper:

\begin{theorem}
A submodule $\clm$ of $\clq_\Theta$ is doubly commuting if and only if
\[
\clm = \eta_1 \clq_{\vp_1} \otimes \cdots \otimes \eta_n \clq_{\vp_n},
\]
where $\eta_j, \vp_j \in H^\infty(\D)$ are inner functions satisfying $\eta_j \vp_j = \theta_j$ for all $j \in I_n$.
\end{theorem}

While this result appears complete, the techniques required to obtain it are fairly involved. In fact, one must undertake a detailed analysis of invariant subspaces, even in the case of classical Jordan blocks. Some of the intermediate results obtained in this process are of independent interest.

A simple extension of the ideas used in the proof of the above theorem also yields representations of all doubly commuting submodules of the mixed space $H^2(\D^{n-r}) \otimes \clq_\Theta$, where $1 \leq r < n$, and $\clq_\Theta = \clq_{\theta_1} \otimes \cdots \otimes \clq_{\theta_r}$ is a Jordan block as above. We emphasize that the origins of these results lie in the earlier work \cite{MR2380073}, which addressed the same characterization problem for doubly commuting submodules of the mixed space $H^2(\D) \otimes \clq_{z^2}$. For a more detailed discussion, see Remark \ref{rem: Nakazi et al}.

The remainder of the paper is organized as follows. In Section \ref{sec: class JB}, we recall some basic facts concerning classical Jordan blocks and compute their defect operators. Section \ref{sec: subm clas JB} is devoted to certain properties of submodules of classical Jordan blocks. In Section \ref{sec: red}, we study  reducing subspaces of Jordan blocks in several variables. Section \ref{sec: submJB} contains the main representation theorems for doubly commuting submodules of Jordan blocks in several variables. In Section \ref{sec: mix sub}, we study such submodules in the setting of mixed Jordan blocks. Finally, Section \ref{sec: unit equiv} addresses the unitary equivalence of the doubly commuting submodules obtained in the preceding sections.

\section{Classical Jordan blocks}\label{sec: class JB}

In this section, we collect some elementary facts about Jordan blocks of $H^2(\D)$. Many of these are known, but we provide full details for completeness. In what follows, we fix a nonconstant inner function $\theta \in H^\infty(\D)$. Our aim is to highlight some distinctive properties of the function $T_z^*\theta$.

The first observation is the following well known fact \cite[Proposition 5.15]{Garcia}:
\begin{equation}\label{eqn: Tztheta}
T_z^{*m} \theta \in \clq_\theta,
\end{equation}
for all $m \geq 1$. Indeed, we have $T_\theta^{*} (T_z^{*m} \theta) = T_z^{*m} T_\theta^* \theta = T_z^{*(m-1)}  T_z^* 1 = 0$. In particular, $T_z^* \theta \in \clq_\theta$. Moreover, it is well-known that $T_z^*\theta$ is a cyclic vector for $S_\theta^*$ (note that $S_\theta^{*m} = T_z^{*m}|_{\clq_\theta}$ for all $m \geq 1$):
\begin{equation}\label{eqn: Tztheta cyclic}
\bigvee_{m=1}^{\infty}T_z^{*m}\theta = \clq_{\theta}.
\end{equation}
To see this, let $f \in \clq_\theta$ satisfying $\langle f, T_z^{*m} \theta \rangle = 0$ for all $m \geq 1$. Define
\[
\clq = \bigvee\{T_z^m \theta, T_z^{*m}\theta: m \in \Z_+\}.
\]
Clearly, $\theta H^2(\D) \subseteq \clq$, and since $f \in \clq_\theta$, it follows that $f \perp \clq$. If $\clq = H^2(\D)$, then $f = 0$ and we are done. Otherwise, since $\clq$ is a quotient module, there exists a nonconstant inner function $\eta \in H^\infty(\D)$ such that $\clq = \clq_\eta$. We are then in the following situation:
\[
\theta \eta \in \theta H^2(\D) \cap \eta H^2(\D) \subseteq \clq_\eta,
\]
which is not possible (note that $\eta H^2(\D) \perp \clq_\eta$), thus yielding that $f = 0$.

We now compute the norm of $T_z^* \theta$. In view of
\begin{equation}\label{eqn: T T*}
T_z T_z^* = I - P_\mathbb{C},
\end{equation}
$P_\mathbb{C}$ being the orthogonal projection of $H^2(\D)$ onto the space of constant functions, and $\|\theta\| = 1$, we have
\[
\|T_z^*\theta\|^2 = \langle T_z T_z^* \theta, \theta \rangle = \langle \theta - \theta(0), \theta \rangle = 1 - \langle \theta(0), \theta \rangle.
\]
Finally, since $\langle \theta(0), \theta \rangle = |\theta(0)|^2$, it follows that
\begin{equation}\label{eqn: norm of Tztheta}
\|T_z^*\theta\|^2 = 1 - |{\theta(0)}|^2.
\end{equation}
Next, we aim to compute defect operators of $S_\theta$. Before proceeding, note that $P_{\clq_\theta} = I - T_\theta T_\theta^*$ and $T_\theta^* 1 = \overline{\theta(0)}$. Hence (also see \cite{HB} and \cite[Lemma 1.1]{MR3272037})
\begin{equation}\label{eqn PQtheta 1}
P_{\clq_\theta} 1 = 1 - \overline{\theta(0)} \theta.
\end{equation}
In what follows, for a pair of vectors $f$ and $g$ in a Hilbert space $\clh$, $f \otimes g$ denotes the rank-one operator defined by $(f \otimes g)h = \langle h, g \rangle f$ for all $h \in \clh$. We are now ready to compute the first defect operator:

\begin{lemma}\label{lem: 1 - SS*}
$I_{\clq_\theta} - S_\theta S_\theta^* = (1 - \overline{\theta(0)} \theta) \otimes (1 - \overline{\theta(0)} \theta)$.
\end{lemma}
\begin{proof}
First, observe that
\[
I_{\clq_\theta} - S_\theta S_\theta^* = I_{\clq_\theta} - P_{\clq_\theta} T_z T_z^*|_{\clq_\theta} = P_{\clq_\theta} P_{\mathbb{C}}|_{\clq_\theta}.
\]
Fix $f \in \clq_\theta$. Note that \eqref{eqn PQtheta 1} implies $P_{\clq_\theta} P_{\mathbb{C}} f = (1 - \overline{\theta(0)} \theta) f(0)$. Also,
\[
f(0) = \langle f, 1 \rangle = \langle f, P_{\clq_\theta} 1 \rangle,
\]
and hence
\[
\begin{split}
(I_{\clq_\theta} - S_\theta S_\theta^*)f & = (1 - \overline{\theta(0)} \theta) \langle f, P_{\clq_\theta} 1 \rangle = (1 - \overline{\theta(0)} \theta) \langle f, 1 - \overline{\theta(0)} \theta \rangle,
\end{split}
\]
where the final identity again follows from \eqref{eqn PQtheta 1}.
\end{proof}

Next, we compute the second defect operator. At this point, we recall that \cite{HB}
\[
\text{rank} (I - S_\theta^* S_\theta) = 1.
\]

\begin{lemma}\label{lem: 1 - S*S}
$I_{\clq_\theta} - S_\theta^* S_\theta = T_z^*\theta \otimes T_z^* \theta$.
\end{lemma}
\begin{proof}
In view of \eqref{eqn: T T*} and $P_{\clq_\theta} \theta = 0$, we compute
\[
S_\theta^* S_\theta (T_z^* \theta) = T_z^* P_{\clq_\theta} T_z (T_z^* \theta) = T_z^* P_{\clq_\theta} (I - P_\mathbb{C}) \theta =  T_z^* P_{\clq_\theta} (\theta - \theta(0)) = - \theta(0) T_z^* P_{\clq_\theta} 1.
\]
By \eqref{eqn PQtheta 1} and the fact that $T_z^*1 = 0$, we have
\[
\begin{split}
(I_{\clq_\theta} - S_\theta^* S_\theta) T_z^* \theta = T_z^* \theta + \theta(0) T_z^* (1 - \overline{\theta(0)} \theta) = T_z^* \theta - |\theta(0)|^2  T_z^* \theta = (1 - |\theta(0)|^2) T_z^* \theta.
\end{split}
\]
By \eqref{eqn: norm of Tztheta}, we know that $\|T_z^*\theta\|^2 = 1 - |{\theta(0)}|^2$, which implies that
\[
(I_{\clq_\theta} - S_\theta^* S_\theta) T_z^* \theta = \|T_z^*\theta\|^2 T_z^* \theta.
\]
In particular, $T_z^* \theta \in \text{ran}(I_{\clq_\theta} - S_\theta^* S_\theta)$. As $\text{rank} (I_{\clq_\theta} - S_\theta^* S_\theta)  = 1$, there exists $c > 0$ such that
\[
I_{\clq_\theta} - S_\theta^* S_\theta = c (T_z^*\theta \otimes T_z^* \theta),
\]
and hence $(I_{\clq_\theta} - S_\theta^* S_\theta) T_z^* \theta = c \|T_z^*\theta\|^2 T_z^* \theta$. Comparing this with $(I_{\clq_\theta} - S_\theta^* S_\theta) T_z^* \theta = \|T_z^*\theta\|^2 T_z^* \theta$, we conclude that $c = 1$, which completes the proof of the lemma.
\end{proof}

The representations of defect operators of classical Jordan blocks obtained above will be useful in the computations that follow.

\section{Submodules of classical Jordan blocks}\label{sec: subm clas JB}

We continue with a nonconstant inner function $\theta \in H^\infty(\D)$. This section discusses some structural properties of submodules of $\clq_\theta$. Some of the results are expected to enhance and complement the existing understanding of classical Jordan blocks.

Recall that a closed subspace $\clw \subseteq \clq_\theta$ is called a submodule of $\clq_\theta$ if $S_\theta \clw \subseteq \clw$. We begin with the well-known description of submodules of $\clq_\theta$ \cite{HB}; these are precisely the subspaces of the form $\eta \clq_\vp$, where $\eta$ and $\vp$ are inner functions satisfying
\[
\theta = \eta \vp.
\]
Indeed, for a submodule $\clw \subseteq \clq_\theta$, since $S_\theta^* = T_z^*|_{\clq_\theta}$, it follows that
\[
T_z^*(\clq_\theta \ominus \clw) \subseteq (\clq_\theta \ominus \clw).
\]
In other words, $\clq_\theta \ominus \clw$ is a Jordan block, and hence, there exists an inner function $\eta \in H^\infty(\D)$ such that
\[
\clq_\theta \ominus \clw = \clq_\eta.
\]
In particular, $\clq_\eta \subseteq \clq_\theta$. Equivalently, $\theta H^2(\D) \subseteq \eta H^2(\D)$, or, equivalently, there exists an inner function $\vp \in H^\infty(\D)$ such that
\[
\theta = \eta \vp.
\]
We also have $\clw = \clq_\theta \ominus \clq_\eta$. Since
\[
(I - T_\theta T_\theta^*) - (I - T_\eta T_\eta^*) = T_\eta T_\eta^* - T_\eta T_\vp T_\vp^* T_\eta^* = T_\eta (I - T_\vp T_\vp^*)T_\eta^*,
\]
it follows that $\clw = \eta \clq_\vp$. The above identity also gives the useful identity
\begin{equation}\label{eqn: P_clm}
P_{\eta \clq_\vp} = T_\eta P_{\clq_\vp} T_\eta^*.
\end{equation}

In \eqref{eqn: Tztheta cyclic}, we noted that $T_z^*\theta$ is a cyclic vector for $S_\theta^*$. In what follows, we show that this function possesses several additional interesting properties.
Given a vector $f$ in a Hilbert space $\clh$, we denote by $\mathbb{C}f$ the one-dimensional subspace of $\clh$ generated by $f$. The operator $P_{\mathbb{C}f}$ refers to the orthogonal projection of $\clh$ onto $\mathbb{C}f$.

\begin{lemma}\label{ran}
Let $\clw$ be a nonzero submodule of $\clq_{\theta}$. Then
\[
P_{\clw}(T_z^*\theta) \neq 0.
\]
Moreover,
\[
P_{\clw}(\mathbb{C}(T_z^*\theta))  = \text{ran} (P_{\clw}P_{\mathbb{C}(T_z^*\theta)}P_{\clw}).
\]
\end{lemma}
\begin{proof}
Represent $\clw$ as $\clw = \eta\clq_{\varphi}$, where $\theta = \eta\varphi$ for inner functions $\eta, \vp \in H^\infty(\D)$. By \eqref{eqn: P_clm}, we have $P_{\clw} = P_{\eta \clq_\vp} = T_\eta P_{\clq_\vp} T_\eta^*$. Using $P_{\clq_\vp} (T_z^*\vp) = T_z^*\vp$, we obtain
\[
P_\clw T_z^* \theta = T_\eta P_{\clq_\vp} T_\eta^* T_z^* \theta = T_\eta P_{\clq_\vp} T_z^* \vp = T_\eta T_z^* \vp.
\]
Since $\vp$ is nonconstant, it follows that $T_z^* \vp$, and hence $\eta T_z^* \vp$ is nonzero. This implies $P_\clw T_z^* \theta \neq 0$. For the remaining part, for notational simplicity, write $T_z^* \theta = f$. We prove that $P_{\clw}P_{\mathbb{C} (T_z^*\theta)}P_{\clw} (P_\clw f) = \alpha (P_\clw f)$ for some nonzero scalar $\alpha$. We have
\[
\begin{split}
P_{\clw}P_{\mathbb{C} f}P_{\clw} (P_\clw f)  = P_{\clw}P_{\mathbb{C} f}P_{\clw} f = P_{\clw}(P_{\mathbb{C} f} T_\eta T_z^* \vp).
\end{split}
\]
Since $P_{\mathbb{C} f}$ is a rank-one orthogonal projection (where $f = T_z^* \theta$), we have
\[
P_{\mathbb{C} f} (T_\eta T_z^* \vp) = \frac{\langle T_\eta T_z^* \vp, T_z^*\theta \rangle}{\|T_z^* \theta\|^2} T_z^*\theta.
\]
Since $\theta = \eta\varphi$, it follows that
\[
\langle T_\eta T_z^* \vp, T_z^*\theta \rangle = \langle T_\eta T_z^* \vp, T_z^* T_\eta \vp \rangle = \langle T_z^* \vp, T_\eta^* T_z^* T_\eta \vp \rangle = \|T_z^* \vp\|^2,
\]
which yields (again, recall that $f = T_z^* \theta$)
\[
P_{\clw}P_{\mathbb{C} f}P_{\clw} (P_\clw T_z^*\theta) = \bigg(\frac{\|T_z^*\varphi\|}{\|T_z^*\theta\|}\bigg)^2 P_{\clw}T_z^*\theta.
\]
Given that $\varphi$ is nonconstant, we also have $\|T_z^*\varphi\|\neq0$. Since $P_{\clw} P_{\mathbb{C}(T_z^*\theta)} P_{\clw}$ is at most of rank one, we conclude that $P_{\clw}(\mathbb{C}(T_z^*\theta))  = \text{ran} (P_{\clw}P_{\mathbb{C}(T_z^*\theta)}P_{\clw})$.
\end{proof}

In the following, we show that $P_{\clw}T_z^*\theta$ is a cyclic vector in the following sense:

\begin{lemma}\label{full}
Let $\clw$ be a submodule of $\clq_{\theta}$. Then
\[
\clw=\bigvee_{m \in \Z_+}P_{\clw}S^{*m}_{\theta} (P_{\clw}T_z^*\theta).
\]
\end{lemma}
\begin{proof}
Set $\cll = \mathbb{C} (P_{\clw}T_z^*\theta)$. Clearly, $\vee_{m \in \Z_+}P_{\clw}S^{*m}_{\theta}\cll\subseteq\clw$. Pick $f \in \clw$ such that $f\perp\vee_{m \in \Z_+}P_{\clw}S^{*m}_{\theta}\cll$. Since $\clw$ is a submodule, it follows that $P_{\clw}S^{*m}_{\theta}P_{\clw} = P_{\clw}S^{*m}_{\theta}$, and hence
\begin{align*}
0=\langle f,P_{\clw}S^{*m}_{\theta}P_{\clw}T_z^*\theta\rangle = \langle f, P_{\clw} S^{*m}_{\theta} T_z^*\theta\rangle = \langle f,S^{*m}_{\theta}T_z^*\theta\rangle = \langle f, T_z^{*(m+1)}\theta\rangle,
\end{align*}
for all $m \in \Z_+$. Since $\vee_{m \in \Z_+}T_z^{*m}\theta = \clq_{\theta}$ (see \eqref{eqn: Tztheta cyclic}), it follows that $f=0$.
\end{proof}

We recall the well-known Fej'{e}r’s theorem: Let $h\in L^1(\T)$ with Fourier expansion 
\[
h \sim \sum_{m \in \Z} \hat{h}(m) e^{imt}.
\]
For $N \in \Z_+$, define the $N$-th Fej\'{e}r mean of the Fourier series of $h$ by
\[
(\sigma_N h)(t)=\sum_{|m|\le N}\Big(1-\frac{|m|}{N+1}\Big) \hat{h}(m) e^{imt}
\]
Then Fej'{e}r’s theorem asserts that $\sigma_Nh\to h$ in $L^1(\T)$.

We also recall a standard consequence of the F. and M. Riesz theorem (equivalently, the boundary uniqueness theorem for Hardy spaces):

\begin{theorem}\label{thm: Riesz}
Let $f_1,f_2\in H^2(\D)$, and suppose that
\[
f_1(e^{it})\overline{f_2(e^{it})}=0,
\]
for a.e. $e^{it} \in \T$. Then either $f=0$ or $g=0$.
\end{theorem}

The following proposition is engrossing, as it parallels the fact that submodules of $H^2(\D)$ cannot be orthogonal.

\begin{proposition}\label{orthogonal}
Let $\clw_1$ and $\clw_2$ be submodules of $\clq_\theta$. If $\clw_1$ is orthogonal to $\clw_2$, then either $\clw_1=\{0\}$ or $\clw_2=\{0\}$.
\end{proposition}
\begin{proof}
Suppose there exists a nonzero element $f_1\in\clw_1$. Pick any $f_2\in\clw_2$ and define
\[
h=f_1\overline{f_2}.
\]
Treat $h \in L^1(\T)$ and write its Fourier expansion as
\[
h \sim \sum_{m \in \Z} \hat{h}(m) e^{imt}.
\]
For any $m \in \mathbb{Z}_+$, we have $S_\theta^m f_2 \in \clw_2$, and hence 
\[
\hat{h}(m)=\int_{\mathbb{T}} f_1(t)\overline{f_2(t)}e^{-imt} d\mu(t) = \langle f_1,z^mf_2\rangle = \langle f_1, S_\theta^m f_2 \rangle = 0.
\]
Similarly, for any $m\le -1$, we have $S_\theta^{-m} f_1 \in \clw_1$, and hence
\[
\hat{h}(m)=\int_{\mathbb{T}} f_1(t)\overline{f_2(t)}e^{-imt} d\mu(t)=\langle z^{-m} f_1, f_2 \rangle = \langle S_\theta^{-m} f_1, f_2 \rangle = 0.
\]
Fej\'{e}r’s Theorem then implies that
\[
h = f_1\overline{f_2}=0,
\]
a.e. on $\T$. Theorem \ref{thm: Riesz} now yields $f_2=0$, and hence we conclude that $\clw_2 = \{0\}$.
\end{proof}

We continue with the cyclicity property. Next, we prove that any submodule of a Jordan block is cyclic under the adjoint of the Jordan block.

\begin{proposition}\label{full space}
Let $\clw$ be a nonzero submodule of $\clq_\theta$. Then
\[
\bigvee_{m \in \Z_+} S^{*m}_\theta \clw=\clq_\theta.
\]
\end{proposition}
\begin{proof}
As $\clw \subseteq \clq_\theta$, it follows that $\bigvee_{m \in \Z_+} S^{*m}_\theta \clw \subseteq \clq_\theta$. If possible, suppose there is a nonzero $h \in \clq_\theta\ominus \bigvee_{m \in \Z_+} S^{*m}_\theta \clw$. Then
\[
\clw_1:=\bigvee_{n=0}^{\infty}S^n_\theta h,
\]
is a nonzero submodule of $\clq_\theta$. For each $g\in \clw$, and $n\in\mathbb{Z}_+$, we have
\[
\langle g, S^n_\theta h\rangle=\langle S^{*n}_\theta g, h\rangle=0,
\]
which shows that $\clw$ and $\clw_1$ are orthogonal, a contradiction to Proposition \ref{orthogonal}. This completes the proof.
\end{proof}

The above result is in line with the classical fact that submodules of $H^2(\mathbb{D})$ are cyclic subspaces for the backward shift.

\section{Reducing submodules}\label{sec: red}

Let $\clq_{\theta_j}$, $j \in I_n$, be a Jordan block of $H^2(\D)$. Let
\[
\clq_\Theta = \clq_{\theta_1} \otimes \cdots \otimes \clq_{\theta_n}.
\]
For each $j \in I_n$, define $S_{\Theta_j} \in \clb(\clq_\Theta)$ by
\[
S_{\Theta_j} = I_{\clq_{\theta_1}} \otimes \cdots I_{\clq_{\theta_{j-1}}} \otimes \underbrace{S_{\theta_j}}_{j-\text{th}} \otimes I_{\clq_{\theta_{j+1}}} \cdots \otimes I_{\clq_{\theta_n}}.
\]
Given a Hilbert space $\clh$, we define $M_j \in \clb(\clh \otimes \clq_\Theta)$, $j \in I_n$, by
\[
M_j = I_{\clh} \otimes S_{\Theta_j}.
\]
Then $M_{\clh, \Theta} := (M_1, \ldots, M_n)$ on $\clh \otimes \clq_\Theta$ is doubly commuting. A closed subspace $\clm$ of $\clh \otimes \clq_\Theta$ is said to be a \textit{reducing submodule} if $M_j^* \clm, M_j \clm \subseteq \clm$; equivalently,
\[
P_\clm M_j = M_j P_\clm,
\]
for all $j \in I_n$. In the following, we describe all reducing submodules of $\clh \otimes \clq_\Theta$.

\begin{proposition}\label{reduce}
A closed subspace $\clm$ of $\clh \otimes \clq_{\Theta}$ is a reducing submodule if and only if there exists a closed subspace $\cll$ of $\clh$ such that
\[
\clm = \cll \otimes \clq_{\Theta}.
\]
\end{proposition}
\begin{proof}
If $\clm = \cll \otimes \clq_{\Theta}$ for some closed subspace $\cll$ of $\clh$, then it is clearly a reducing submodule (note that $S_\Theta \in \clb(\clq_\Theta)$). Now, assume that $\clm$ is a reducing submodule, that is, $P_{\clm} M_i^* = M_i^*P_{\clm}$ for all $i \in I_n$. This implies
\[
P_\clm \prod_{i=1}^n (I - M_i^* M_i) = \prod_{i=1}^n (I - M_i^* M_i) P_\clm.
\]
By Lemma \ref{lem: 1 - S*S}, we know that
\[
\prod_{i=1}^n (I - M_i^* M_i) = I_\clh \otimes (T_z^* \theta_1 \otimes T_z^* \theta_1) \otimes \cdots \otimes (T_z^* \theta_n \otimes T_z^* \theta_n).
\]
Since
\[
P_{\mathbb{C} T_z^* \theta_i} = \frac{1}{\|{T^*_z}\theta_i\|^2} (T_z^* \theta_i \otimes T_z^* \theta_i),
\]
for all $i \in I_n$, it follows that
\[
P_\clm (I_\clh \otimes P_{\mathbb{C} T_z^* \theta_1} \otimes \cdots \otimes P_{\mathbb{C} T_z^* \theta_n}) = (I_\clh \otimes P_{\mathbb{C} T_z^* \theta_1} \otimes \cdots \otimes P_{\mathbb{C} T_z^* \theta_n}) P_\clm.
\]
In particular, the above shows that two orthogonal projections commute, and therefore their product is also an orthogonal projection onto $\tilde \clm$, where
\[
\tilde \clm = \clm\cap(\clh \otimes (\mathbb{C} T_z^* \theta_1) \otimes \cdots \otimes (\mathbb{C} T_z^* \theta_n)).
\]
Therefore,
\[
P_{\tilde \clm} = P_\clm (I_\clh \otimes P_{\mathbb{C} T_z^* \theta_1} \otimes \cdots \otimes P_{\mathbb{C} T_z^* \theta_n}) = (I_\clh \otimes P_{\mathbb{C} T_z^* \theta_1} \otimes \cdots \otimes P_{\mathbb{C} T_z^* \theta_n}) P_\clm.
\]
Clearly, there exists a closed subspace $\cll \subseteq \clh$ such that
\[
\tilde \clm = \cll \otimes (\mathbb{C} T_z^* \theta_1) \otimes \cdots \otimes (\mathbb{C} T_z^* \theta_n).
\]
By the cyclicity property \eqref{eqn: Tztheta cyclic}, it follows that
\[
\bigvee_{k \in \Z_+^n} M^{*k} \tilde{\clm} = \cll \otimes \clq_\Theta.
\]
On the other hand, since $\tilde{\clm}\subseteq\clm$ and $\clm$ is a reducing submodule, we get
\[
\bigvee\limits_{k \in\mathbb{Z}_+^n} M^{*k}\tilde{\clm}\subseteq\clm.
\]
To prove that this inclusion is actually an equality, we first observe that
\[
\clq_\Theta = \bigvee_{k \in \Z_+^n} S_\Theta^{*k} (T_z^* \theta_1 \otimes \cdots \otimes T_z^* \theta_n).
\]
This implies
\[
\clh \otimes \clq_\Theta = \bigvee_{h \in \clh, k \in \Z_+^n} h \otimes (S_\Theta^{*k} (T_z^* \theta_1 \otimes \cdots \otimes T_z^* \theta_n)).
\]
For each $h\in\clh$, call
\[
h_\Theta = (h\otimes T_z^*\theta_1\otimes\cdots\otimes T_z^*\theta_n).
\]
Then
\begin{align*}
P_{\clm}(h \otimes (S_\Theta^{*k} (T_z^* \theta_1 \otimes \cdots \otimes T_z^* \theta_n))) & = P_{\clm} M^{*k} h_\Theta
\\
& = M^{*k} P_{\clm} h_\Theta
\\
& = M^{*k} P_\clm (I_\clh \otimes P_{\mathbb{C} T_z^* \theta_1} \otimes \cdots \otimes P_{\mathbb{C} T_z^* \theta_n}) h_\Theta
\\
& = M^{*k} P_{\tilde{\clm}} h_\Theta,
\end{align*}
which implies that $\clm\subseteq \bigvee_{k \in \Z_+^n} M^{*k} \tilde{\clm}$. Therefore, $\clm = \bigvee_{k \in \Z_+^n} M^{*k} \tilde{\clm} = \cll \otimes \clq_{\Theta}$ for some closed subspace $\cll\subseteq\clh$.
\end{proof}

We say that a closed subspace $\clm$ of $\clh \otimes \clq_{\Theta}$ is invariant under $M_{\clh, \Theta}$ if $M_i \clm \subseteq \clm$ for all $i \in I_n$. It is thus natural to refer to $\clm$ as a submodule of $\clh \otimes \clq_\Theta$.

\begin{proposition}\label{model reducing common}
Let $\clm$ be a closed subspace of $\clh \otimes \clq_{\Theta}$.
Define
\[
\tilde \clm = \bigvee\limits_{k \in \Z_+^n} M^{*k} \clm.
\]
If $\clm$ is invariant under $M_{\clh, \Theta}$, then $\tilde \clm$ is a reducing submodule of $\clh \otimes \clq_\Theta$.
\end{proposition}
\begin{proof}
Clearly, $\tilde \clm$ is $M^*_i$-invariant, $i \in I_n$. Fix an $i\in I_n$. It is enough to prove that $\tilde \clm$ is $M_i$-invariant. Since $(M_1, \ldots, M_n)$ doubly commutes, for $k = (k_1, \ldots, k_n) \in \Z_+^n$, we have
\[
M_i M^{*k} = M^{* k - k_ie_i} M_i M_i^{*k_i}.
\]
Therefore, it is enough to prove that
\[
M_i M_i^{*l} \clm \subseteq \tilde \clm,
\]
for all $l \ge 1$. To prove this, we first define
\[
\clh_i:=\clq_{\theta_i}\otimes\clh\otimes\clq_{\theta_1}\otimes\cdots\otimes\clq_{\theta_{i-1}}\otimes\clq_{\theta_{i+1}}\otimes\cdots\clq_{\theta_n},
\]
and identify $\clh\otimes\clq_\Theta$ with $\clh_i$ via the unitary $U: \clh\otimes\clq_\Theta \raro \clh_i$ defined by
\[
U(h\otimes q_1\otimes\cdots\otimes q_n)=q_i\otimes h\otimes q_1\otimes\cdots\otimes q_{i-1}\otimes q_{i+1}\otimes\cdots q_n,
\]
for all $h\in\clh$ and $q_j\in \clq_{\theta_j}$, and $j\in I_n$. It is now enough to prove that
\[
A_i A_i^{*l} (U\clm) \subseteq  U\tilde \clm,
\]
for all $l \ge 1$, where
\[
A_i:=UM_iU^*=S_{\theta_i}\otimes I_\clh\otimes I_{\clq_{\theta_1}}\otimes\cdots\otimes I_{\clq_{\theta_{i-1}}}\otimes I_{\clq_{\theta_{i+1}}}\otimes\cdots I_{\clq_{\theta_n}}.
\]
To this end, fix $l \geq 1$ and $g \in \clm$. Then $Ug\in U \clm \subseteq \clh_i$. Let $\{f_m:m\in \mathbb{N}\}$ and $\{e_m:m\in\mathbb{N}\}$ be orthonormal bases for $\clq_{\theta_i}$ and $\clh\otimes\clq_{\theta_1}\otimes\cdots\otimes\clq_{\theta_{i-1}}\otimes\clq_{\theta_{i+1}}\otimes\cdots\clq_{\theta_n}$, respectively. There exist scalars $\{c_{jm}\}_{j,m=1}^\infty$ such that
\[
Ug= \sum _{j,m=1}^{\infty} c_{jm}f_j\otimes e_m
=\sum _{m=1}^{\infty} \Big(\sum _{j=1}^{\infty}c_{jm} f_j\Big)\otimes e_m=\sum _{m=1}^{\infty} g_m\otimes e_m,
\]
where
\[
g_m := \sum _{j=1}^{\infty}c_{jm}f_j,
\]
for all $m \geq 1$. Let $r \geq 1$ be arbitrary. Since $\clm$ is $M_i$-invariant, it follows that $U \clm$ is $A_i$-invariant, and
\begin{align}\label{inv}
A^r_i Ug = \sum_{m=1}^{\infty} S^r_{\theta_i}g_m \otimes e_m \in U \clm.
\end{align}
Since $\tilde \clm$ is $M_i^*$-invariant, $U\tilde \clm$ is $A_i^*$-invariant, and hence
\[
A_i^* A^r_i Ug \in U\tilde \clm.
\]
Also, by Lemma \ref{lem: 1 - S*S}, we know that $I_{\clq_{\theta_i}} - S^*_{\theta_i} S_{\theta_i} = T_z^*\theta_i \otimes T_z^*\theta_i$. Therefore,
\begin{align*}
A_i^* A^r_i Ug & = \sum_{m=1}^{\infty} (S^*_{\theta_i}S_{\theta_i}) S^{r-1}_{\theta_i} g_m \otimes e_m = \sum_{m=1}^{\infty} (S^{r-1}_{\theta_i}g_m - \langle S^{r-1}_{\theta_i}g_m, T_z^*\theta_i \rangle T_z^*\theta_i)\otimes e_m.
\end{align*}
As $A_i^* A^r_i Ug \in U \tilde \clm$ and $\sum_{m=1}^{\infty} S^{r-1}_{\theta_i}g_m \otimes e_m \in U\clm \subseteq U\tilde{\clm}$ (see \eqref{inv} above), it follows that
\begin{align*}
U \tilde \clm & \ni \sum_{m=1}^{\infty} \langle S^{r-1}_{\theta_i} g_m, T_z^*\theta_i\rangle T_z^*\theta_i \otimes e_m = T_z^*\theta_i \otimes \sum_{m=1}^{\infty} \langle g_m, T_z^{*r} \theta_i \rangle e_m.
\end{align*}
Moreover, as $A_i^* U\tilde{\clm} \subseteq U \tilde{\clm}$ and $\bigvee_{m \geq 1} T^{*m}_z \theta_i = \clq_{\theta_i}$ (see \eqref{eqn: Tztheta cyclic}), the above implies
\begin{align}\label{1}
f \otimes \sum_{m=1}^{\infty} \langle g_m, T_z^{*r} \theta_i \rangle e_m \in U\tilde{\clm},
\end{align}
for all $f\in\clq_{\theta_i}$ and $r \geq 1$. Now we turn to proving that
\[
A_i A_i^{*l} Ug \in U \tilde{\clm},
\]
where $l$ is a fixed natural number. Recall from Lemma \ref{lem: 1 - SS*} that $I_{\clq_{\theta_i}} - S_{\theta_i}S_{\theta_i}^* = \hat{\theta}_i \otimes \hat{\theta}_i$, where
\[
\hat{\theta}_i = 1-\overline{\theta_i(0)}\theta_i,
\]
for all $i \in I_n$. In particular, \eqref{1} implies
\begin{equation}\label{eqn: hir}
h_{ir} := \hat{\theta}_i \otimes\sum_{m=1}^{\infty} \langle g_m, T_z^{*r} \theta_i \rangle e_m \in U \tilde{\clm},
\end{equation}
for all $r \geq 1$. Writing $A_i A_i^{*l} = (A_i A_i^*)A_i^{*l-1}$, we compute
\begin{align*}
A_i A_i^{*l} Ug & = \sum\limits_{m=1}^{\infty} (S_{\theta_i} S^*_{\theta_i}) S^{*{l-1}}_{\theta_i} g_m \otimes e_m
\\
& = A^{*{l-1}}Ug - \sum\limits_{m=1}^{\infty}\langle S^{*{l-1}}_{\theta_i}g_m, \hat{\theta}_i \rangle \hat{\theta}_i \otimes e_m
\\
& = A^{*{l-1}}Ug - \hat{\theta}_i \otimes \sum\limits_{m=1}^{\infty}\langle S^{*{l-1}}_{\theta_i}g_m, \hat{\theta}_i \rangle e_m.
\end{align*}
Since $A^{*{l-1}}Ug \in U \tilde{\clm}$, the assertion that $A_i A_i^{*l} Ug \in U \tilde{\clm}$  will be proved if we can show that
\[
h:= \hat{\theta}_i \otimes \sum\limits_{m=1}^{\infty} \langle S^{*{l-1}}_{\theta_i}g_m, \hat{\theta}_i \rangle e_m \in U\tilde{\clm}.
\]
To this end, we recall another known fact about model spaces:
\[
f = \sum_{p=1}^{\infty} \langle f, T_z^{*p}\theta \rangle T_z^{*p} \theta,
\]
and
\[
\|f\|^2=\sum_{p=1}^{\infty} |\langle f, T_z^{*p}\theta \rangle|^2
\]
for all $f \in \clq_\theta$ \cite[Proposition 5.15]{Garcia}. In view of this, for each $m \geq 1$, we compute
\[
\begin{split}
\langle S^{*{l-1}}_{\theta_i}g_m, \hat{\theta}_i \rangle & = \langle g_m, S^{{l-1}}_{\theta_i} \hat{\theta}_i \rangle
\\
& = \langle g_m, \sum_{r=1}^\infty \langle S^{{l-1}}_{\theta_i} \hat{\theta}_i, T_z^{*r} \theta_i \rangle T_z^{*r} \theta_i \rangle
\\
& = \sum_{r=1}^{\infty} \overline{\langle S^{{l-1}}_{\theta_i} \hat{\theta}_i, T_z^{*r} \theta_i \rangle} \langle g_m, T_z^{*r} \theta_i \rangle
\\
& = \sum_{r=1}^{\infty} c_{ilr} \langle g_m, T_z^{*r} \theta_i \rangle,
\end{split}
\]
where
\[
c_{ilr} = \overline{\langle S^{{l-1}}_{\theta_i} \hat{\theta}_i, T_z^{*r} \theta_i \rangle}.
\]
Therefore,
\begin{align*}
h & =  \hat{\theta}_i \otimes \sum\limits_{m=1}^{\infty} \langle S^{*{l-1}}_{\theta_i}g_m, \hat{\theta}_i \rangle e_m =  \hat{\theta}_i \otimes \sum\limits_{m=1}^{\infty} \sum_{r=1}^{\infty} c_{ilr} \langle g_m, T_z^{*r} \theta_i \rangle  e_m
\end{align*}
Let
\[
b_r:=\Big\|\sum_{m=1}^{\infty}\langle g_m, T_z^{*r}\theta_i\rangle e_m\Big\|.
\]	
We will prove that $(b_r)_{r\ge 1}\in l^2$. By Tonelli's theorem,
\begin{align*}
\sum_{r=1}^{\infty}|b_r|^2 =\sum_{r=1}^{\infty}\Big\|\sum_{m=1}^{\infty}\langle g_m, T_z^{*r}\theta_i\rangle e_m\Big\|^2 =\sum_{r=1}^{\infty}\sum_{m=1}^{\infty}|\langle g_m, T_z^{*r}\theta_i\rangle|^2 =\sum_{m=1}^{\infty}\sum_{r=1}^{\infty}|\langle g_m, T_z^{*r}\theta_i\rangle|^2,
\end{align*}
that is,
\begin{align*}
\sum_{r=1}^{\infty}|b_r|^2 =\sum_{m=1}^{\infty}\|g_m\|^2=\|g\|^2.
\end{align*}
Since
\[
\sum_{r=1}^{\infty}|c_{ilr}|^2=\|S^{{l-1}}_{\theta_i} \hat{\theta}_i\|^2,
\]
Cauchy- Schwarz inequality implies
\[
\sum_{r=1}^{\infty} |c_{ilr}|\Big\|\sum\limits_{m=1}^{\infty} \langle g_m, T_z^{*r} \theta_i \rangle  e_m \Big\|<\infty.
\]
Let $v_r:=\sum\limits_{m=1}^{\infty} \langle g_m, T_z^{*r} \theta_i \rangle  e_m$. Then,
\[
\Big\|\sum_{r=1}^{\infty}c_{ilr}v_r\Big\|\le \sum_{r=1}^{\infty}|c_{ilr}|\|v_r\|<\infty
\]
gives that (recall that $i$ and $l$ are fixed)
\[
f:=\sum_{r=1}^{\infty}c_{ilr}v_r\in \clh\otimes\clq_{\theta_1}\otimes\cdots\otimes\clq_{\theta_{i-1}}\otimes\clq_{\theta_{i+1}}\otimes\cdots\clq_{\theta_n}.
\]
Since $\{e_m:m\in \mathbb{N}\}$ is an orthonormal basis for $\clh\otimes\clq_{\theta_1}\otimes\cdots\otimes\clq_{\theta_{i-1}}\otimes\clq_{\theta_{i+1}}\otimes\cdots\clq_{\theta_n}$, we get
\[
\langle f,e_m\rangle=\Big\langle \sum_{r=1}^{\infty}c_{ilr}v_r,e_m\Big\rangle=\sum_{r=1}^{\infty}c_{ilr}\langle v_r,e_m \rangle=\sum_{r=1}^{\infty}c_{ilr}\langle g_m,T_z^{*r}\theta_i \rangle,
\]
which further yields
\[
f=\sum\limits_{m=1}^{\infty}\langle f,e_m\rangle e_m=\sum\limits_{m=1}^{\infty}\sum_{r=1}^{\infty}c_{ilr}\langle g_m,T_z^{*r}\theta_i \rangle e_m.
\]
Therefore,
\[
\sum_{r=1}^{\infty}c_{ilr}\sum\limits_{m=1}^{\infty} \langle g_m, T_z^{*r} \theta_i \rangle  e_m=\sum_{r=1}^{\infty}c_{ilr}v_r=f=\sum\limits_{m=1}^{\infty}\sum_{r=1}^{\infty}c_{ilr}\langle g_m,T_z^{*r}\theta_i \rangle e_m.
\]
Finally, we have
\begin{align*}
h & =  \hat{\theta}_i \otimes \sum\limits_{m=1}^{\infty} \sum_{r=1}^{\infty} c_{ilr} \langle g_m, T_z^{*r} \theta_i \rangle  e_m
\\
& =  \hat{\theta}_i \otimes\sum\limits_{r=1}^{\infty} c_{ilr} \Big(\sum_{m=1}^{\infty} \langle g_m, T_z^{*r} \theta_i \rangle e_m \Big)
\\
& = \sum\limits_{r=1}^{\infty} c_{ilr} \Big(\hat{\theta}_i \otimes\sum_{m=1}^{\infty} \langle g_m, T_z^{*r} \theta_i \rangle  e_m \Big)
\\
& = \sum\limits_{r=1}^{\infty} c_{ilr} h_{ir},
\end{align*}
where $h_{ir}$'s are as in \eqref{eqn: hir}. Since all $h_{ir}$ are in $U\tilde{\clm}$, we finally conclude that $h \in U\tilde{\clm}$.
\end{proof}

The cyclicity property outlined above will play a key role in determining the structure of submodules of Jordan blocks.

\section{Submodules of Jordan blocks}\label{sec: submJB}
	
This section presents representations of doubly commuting submodules of Jordan blocks of $H^2(\D^n)$. We begin with yet another result on reducing subspaces. Fix inner functions $\eta$ and $\vp$ from $H^\infty(\D)$ with $\theta=\eta\vp$. Given a Hilbert space $\clh$, we consider the space $\eta\clq_{\varphi} \otimes \clh$ and the operator
\[
S:= S_\theta|_{\eta\clq_{\varphi}} \otimes I_\clh,
\]
on $\eta\clq_{\varphi} \otimes \clh$. The following result reveals the structure of reducing subspaces of $S$. This result is essentially Proposition \ref{reduce}, but interpreted in light of the unitary equivalence of submodules of Jordan blocks with the Jordan block itself \cite[page 38]{HB}.

\begin{proposition}\label{subset reducing}
Let $\clm$ be a closed subspace of $\eta\clq_{\varphi} \otimes \clh$. Then $\clm$ reduces $S$ if and only if there exists a closed subspace $\cle$ of $\clh$ such that
\[
\clm = \eta\clq_{\varphi} \otimes\cle.
\]
\end{proposition}
\begin{proof}
Recall from \eqref{eqn: P_clm} that $P_{\eta\clq_{\varphi}} = T_\eta P_{\clq_{\varphi}} T_\eta^*$. Therefore, for $f\in\clq_{\varphi}$, we have
\begin{align*}
S_\theta|_{\eta\clq_{\varphi}} (\eta f) = P_{\eta\clq_{\varphi}}S_{\theta}|_{\eta\clq_{\varphi}} (\eta f) = T_\eta P_{\clq_\varphi} T^*_\eta z\eta f =\eta P_{\clq_\varphi} z f =\eta S_\varphi f.
\end{align*}
Observe that $U(f\otimes h)=\eta f\otimes h$, for all $f \in \clq_\vp$ and $h \in \clh$, defines a unitary operator $U:\clq_\varphi\otimes\clh \to \eta\clq_{\varphi}\otimes\clh$. Using the above identity, we also have:
\begin{align*}
U(S_\varphi\otimes I_\clh)(f\otimes h)= (\eta S_\varphi f)\otimes h = S_{\theta}|_{\eta\clq_{\varphi}} (\eta f) \otimes h = S (\eta f\otimes h)= S U(f\otimes h),
\end{align*}
for all $f \in \clq_\vp$ and $h \in \clh$, and hence
\[
U(S_\varphi\otimes I_\clh) = S U.
\]
Now we consider the closed subspace $\clm \subseteq \eta\clq_{\varphi} \otimes \clh$. If $\clm = \eta\clq_{\varphi} \otimes\cle$ for some closed subspace $\cle \subseteq \clh$, then clearly $\clm$ reduces $S$. For the reverse direction, assume that $\clm$ reduces $S$. By Proposition \ref{reduce}, there exists a closed subspace $\cle \subseteq \clh$ such that $U^* \clm = \clq_\varphi\otimes\cle$. Therefore, $\clm=U(\clq_\varphi\otimes\cle)=\eta\clq_\varphi\otimes\cle$, which completes the proof of the proposition.
\end{proof}

With this, we are now ready to represent doubly commuting submodules of Jordan blocks:

\begin{theorem}\label{main 0}
Let $\clm$ be a submodule of $\clq_\Theta$. Then $\clm$ is doubly commuting if and only if there exist submodules $\clw_{i}$ of $\clq_{\theta_i}$ such that
\[
\clm=\clw_1\otimes\cdots\otimes\clw_n.
\]
\end{theorem}
\begin{proof}
Suppose $\clm$ is a doubly commuting submodule, that is,
\[
S_{\Theta_i}|_{\clm} (S_{\Theta_j}|_{\clm})^* = (S_{\Theta_j}|_{\clm})^* S_{\Theta_i}|_{\clm},
\]
for all $i > j$. Define
\[
\clm_1 : = \bigvee \{S_{\Theta_2}^{* m_2} \cdots S_{\Theta_n}^{* m_n} \clm: m_2, \ldots, m_n \in \Z_+\}.
\]
Proposition \ref{model reducing common} implies that $\clm_{1}$ reduces $(S_{\Theta_2}, \ldots, S_{\Theta_n})$. Apply Proposition \ref{reduce} with $\clh = \clq_{\theta_1}$ and $\clq_\Theta$ as $\clq_{\theta_2} \otimes \cdots \otimes \clq_{\theta_n}$ to find a closed subspace $\clw_1 \subseteq \clq_{\theta_1}$ such that
\[
\clm_1=\clw_1\otimes\clq_{\theta_2}\otimes\cdots\otimes\clq_{\theta_n}.
\]
Since $\clm_1$ is $S_{\Theta_1}$-invariant, it follows that $\clw_1$ is an $S_{\theta_1}$-invariant subspace of $\clq_{\theta_1}$, that is, $\clw_1$ is a submodule of $\clq_{\theta_1}$. We claim that $\clm$ reduces $S_{\Theta_1}|_{\clm_1}$, that is,
\[
P_{\clm}(S_{\Theta_1}|_{\clm_1})=(S_{\Theta_1}|_{\clm_1})P_{\clm}.
\]
To this end, we first note that, by the doubly commutativity property of $\clm$, we have
\[
S_{\Theta_1}|_{\clm} (S_{\Theta_j}|_{\clm})^* = (S_{\Theta_j}|_{\clm})^* S_{\Theta_1}|_{\clm},
\]
for all $j =2, \ldots, n$. Therefore, for each $m_i \in \Z_+$, $i=2, \ldots, n$, we have
\begin{align*}
S_{\Theta_1}|_{\clm}\prod\limits_{i=2}^{n}P_{\clm}S_{\Theta_i}^{*{m_i}}|_{\clm} & = \bigg(\prod\limits_{i=2}^{n}P_{\clm}S_{\Theta_i}^{*{m_i}}|_{\clm}\bigg)S_{\Theta_1}|_{\clm}
\\
&=P_{\clm}\bigg(\prod\limits_{i=2}^{n}S_{\Theta_i}^{*{m_i}} S_{\Theta_1}\bigg)|_{\clm}
\\
& = P_{\clm} \bigg(S_{\Theta_1} \prod\limits_{i=2}^{n}S_{\Theta_i}^{*{m_i}}\bigg)|_{\clm},
\end{align*}
which implies
\begin{align*}
P_{\clm}(S_{\Theta_1}|_{\clm_1})\prod\limits_{i=2}^{n}S_{\Theta_i}^{*{m_i}}|_{\clm} & =P_\clm S_{\Theta_1} \prod\limits_{i=2}^{n}S_{\Theta_i}^{*{m_i}}|_{\clm}
\\
& = S_{\Theta_1}|_{\clm}\prod\limits_{i=2}^{n}P_{\clm}S_{\Theta_i}^{*{m_i}}|_{\clm}
\\
&=S_{\Theta_1}P_{\clm}\prod\limits_{i=2}^{n}S_{\Theta_i}^{*{m_i}}|_{\clm}.
\end{align*}
Since $\clm\subseteq\clm_{1}$, we have $P_\clm=P_{\clm_{1}}P_\clm$, and hence
\[
P_{\clm}(S_{\Theta_1}|_{\clm_1})\prod\limits_{i=2}^{n} S_{\Theta_i}^{*{m_i}}f  = S_{\Theta_1} P_{\clm_{1}}P_{\clm}\prod\limits_{i=2}^{n} S_{\Theta_i}^{*{m_i}}f= (S_{\Theta_1}|_{\clm_{1}})P_{\clm}\prod\limits_{i=2}^{n} S_{\Theta_i}^{*{m_i}}f,
\]
for all $f \in \clm$. This proves the claim that $P_{\clm}(S_{\Theta_1}|_{\clm_1})=(S_{\Theta_1}|_{\clm_1})P_{\clm}$, that is, $\clm$ reduces $S_{\Theta_1}|_{\clm_1}$ on $\clm_1 = \clw_1 \otimes (\clq_{\theta_2} \otimes \cdots \otimes \clq_{\theta_n})$. By Proposition \ref{subset reducing}, there exists a closed subspace $\cle_1 \subseteq \clq_{\theta_2}\otimes\cdots\otimes\clq_{\theta_n}$ such that
\[
\clm=\clw_1\otimes\cle_1.
\]
Since $\clm$ is $(S_{\Theta_2},\ldots,S_{\Theta_n})$-invariant, it follows that $\cle_{1}$ is a submodule of $\clq_{\theta_2} \otimes \cdots \otimes \clq_{\theta_n}$. Moreover, note that
\[
P_\clm=P_{\clw_1}\otimes P_{\cle_{1}}.
\]
Using the fact that $\clm$ doubly commutes, for each $2\le r<l\le n$, we have
\[
S_{\Theta_r}|_\clm P_\clm S_{\Theta_l}^*|_\clm = P_\clm S_{\Theta_l}^*|_\clm S_{\Theta_r}|_\clm,
\]
and hence,
\[
(I_{\clw_1}\otimes S_{\Theta_r}|_{\cle_1}) (I_{\clw_1}\otimes P_{\cle_1} S_{\Theta_l}^*|_{\cle_1}) = (I_{\clw_1}\otimes P_{\cle_1}S_{\Theta_l}^*|_{\cle_1}) (I_{\clw_1}\otimes S_{\Theta_r}|_{\cle_1}).
\]
Therefore, it follows that
\[
S_{\Theta_r}|_{\cle_1}P_{\cle_1} S_{\Theta_l}^*|_{\cle_1} = P_{\cle_1} S_{\Theta_l}^*|_{\cle_1} S_{\Theta_r}|_{\cle_1},
\]
which proves that $\cle_{1}\subseteq \clq_{\theta_2}\otimes\cdots\otimes\clq_{\theta_n}$ is a doubly commuting submodule. Now, we focus on the doubly commuting submodule $\cle_1$ of the Jordan block $\clq_{\theta_2}\otimes\cdots\otimes\clq_{\theta_n}$, and apply the preceding method to find a submodule $\clw_2$ of $\clq_{\theta_2}$ and a doubly commuting submodule $\cle_2$ of $\clq_{\theta_3} \otimes \cdots \otimes \clq_{\theta_n}$ such that
\[
\cle_1=\clw_2\otimes\cle_2.
\]
Continuing in this manner, we ultimately obtain
\[
\clm=\clw_1\otimes\cdots\otimes\clw_n,
\]
where each $\clw_i$ is a submodule of $\clq_{\theta_i}$, $i \in I_n$. The converse follows from the fact that
\[
S_{\Theta_i}|_{\clm} = I_{\clw_1} \otimes \cdots\otimes S_{\theta_i}|_{\clw_i}\otimes\cdots\otimes I_{\clw_n},
\]
which clearly doubly commutes.
\end{proof}

The proof of the above result does not allow us to replace some of the Jordan blocks with the full Hardy space $H^2(\D)$. However, with a slight modification of the argument and the use of representations of doubly commuting submodules, it is possible to determine the doubly commuting submodules of such mixed spaces. This will be addressed in the next section.

\section{Mixed submodules}\label{sec: mix sub}

Fix $1 \leq r < n$. This section considers a class of closed subspaces of $H^2(\D^n)$ that are invariant under $n-r$ variables and star-invariant under the other $r$ variables. We present the doubly commuting invariant subspaces of such mixed subspaces of $H^2(\D^n)$. In spirit, the idea of the results and proofs has already been established in previous sections. However, the aim here is twofold: first, to remove the condition that all spaces in the tensor product are Jordan blocks; and second, to both recover and generalize the results of \cite{MR2380073}, which were obtained in the two-variable case.

Recall that \textit{submodules} of $H^2(\D^n)$ are closed subspaces $\cls$ of $H^2(\D^n)$ such that
\[
z_i \cls \subseteq \cls,
\]
for all $i \in I_n$. The following result gives us representations of reducing subspaces of vector-valued submodules of $H^2(\D^n)$:

\begin{proposition}\label{doubly red}
Let $\cls$ be a submodule of $H^2(\D^n)$, $\clh$ be a Hilbert space, and let
\[
\cls_\clh:=\cls\otimes\clh.
\]
Suppose $T_{\cls_\clh}:=((T_{z_1} \otimes I_\clh)|_{\cls_{\clh}}, \ldots, (T_{z_n} \otimes I_\clh)|_{\cls_{\clh}})$ is doubly commuting. Then there exists an inner function $\varphi\in H^\infty(\D^n)$ such that
\[
\cls=\varphi H^2(\D^n).
\]
Moreover, if $\clm$ is a closed subspace of $\cls_\clh$, then $\clm$ reduces $T_{\cls_\clh}$ if and only if there exists a closed subspace $\cle \subseteq \clh$ such that
\[
\clm=\cls\otimes\cle.
\]
\end{proposition}
\begin{proof}
Since $T_i|_{\cls_{\clh}}=T_{z_i}|_\cls\otimes I_\clh$, $i \in I_n$, and $T_{\cls_\clh}$ is doubly commuting, it follows that $(T_{z_1}|_{\cls},\ldots,T_{z_n}|_{\cls})$ on $\cls$ is doubly commuting. In other words, $\cls$ is a doubly commuting submodule of $H^2(\D^n)$, and consequently, there exists an inner function $\varphi\in H^\infty(\D^n)$ such that $\cls=\varphi H^2(\D^n)$ (see \cite{MR3151275, MR3117337}). There exists a unitary operator $U: \vp H^2(\D^n) \raro H^2(\D^n)$ (namely, $U (\vp f) = f$ for all $f \in H^2(\D^n)$) such that $U T_{z_i}|_\cls = T_{z_i} U$ for all $i \in I_n$. Define
\[
\tilde{U}: \vp H^2(\D^n) \otimes \clh=\cls \otimes \clh \raro H^2(\D^n) \otimes \clh,
\]
by $\tilde{U} = U \otimes I_\clh$. The result now follows from the fact that the reducing subspaces of $H^2(\D^n) \otimes \clh$ are precisely those of the form $H^2(\D^n) \otimes \cle$, where $\cle$ is a closed subspace of $\clh$.
\end{proof}

Given $r$ nonconstant inner functions $\{\theta_i\}_{i=1}^r \subseteq H^\infty(\D)$, define the mixed space
\[
\clh^2(\Theta) = H^2(\D^{n-r}) \otimes \clq_\Theta,
\]
where $\clq_\Theta = \clq_{\theta_1} \otimes \cdots \otimes \clq_{\theta_r}$. For each $i \in I_{n-r}$, and $j \in I_r$, define model operators $T_i$ and $S_j$ on $\clh^2(\Theta)$ by
\[
T_i = T_{z_i} \otimes I_{\clq_\Theta} \mbox{ and } S_j = I_{H^2(\D^{n-r})} \otimes S_{\Theta_j}.
\]
We write the $n$-tuple of model operators as
\[
M(\Theta) :=(T_1, \ldots, T_{n-r}, S_1, \ldots, S_r).
\]
As usual, a closed subspace $\clm$ of $\clh^2(\Theta)$ is called a submodule if $T_i \clm, S_j \clm \subseteq \clm$ for all $i$ and $j$. A submodule $\clm$ of $\clh^2(\Theta)$ is called doubly commuting if the tuple
\[
M(\Theta)|_\clm :=(T_1|_\clm, \ldots, T_{n-r}|_\clm, S_1|_\clm, \ldots, S_r|_\clm),
\]
on $\clm$ is doubly commuting.

We will need the following result, which follows directly from Proposition \ref{reduce}: Let $1\le r< n$. Then a closed subspace $\clm_{1}$ of $\clh^2(\Theta)$ reduces $(S_1,\ldots,S_r)$ if and only if there exists a closed subspace $\clw$ of $H^2(\D^{n-r})$ such that
\begin{equation}\label{mixed model red}
\clm_{1}=\clw\otimes\clq_{\theta_1}\otimes\cdots\otimes\clq_{\theta_r}.
\end{equation}

We now present characterizations of doubly commuting submodules, using the notation introduced above.

\begin{theorem}\label{main 1}
A submodule $\clm$ of $\clh^2(\Theta)$ is doubly commuting if and only if there exist an inner function $\varphi\in H^\infty(\D^{n-r})$ and submodules $\clw_{i}$ of $\clq_{\theta_i}$, $i \in I_r$, such that
\[
\clm=\varphi H^2(\D^{n-r}) \otimes (\clw_1\otimes\cdots\otimes\clw_r).
\]
\end{theorem}
\begin{proof}
Suppose $\clm$ is a doubly commuting submodule. Following the proof of Theorem \ref{main 0}, define
\[
\clm_1:=\bigvee\limits_{k \in \mathbb{Z}_+^r} S^{*k}\clm.
\]
By Proposition \ref{model reducing common}, we know that $\clm_{1}$ is a $(S_1,\ldots,S_r)$-reducing subspace of $\clh^2(\Theta)$. By \eqref{mixed model red}, there exists a closed subspace $\clw$ of $H^2(\D^{n-r})$ such that
\[
\clm_{1}=\clw \otimes \clq_{\Theta}.
\]
Note that since $\clm$ is $(T_{1},\ldots,T_{{n-r}})$-invariant, the same holds for $\clm_{1}$. We follow the same argument as in the proof of the reducing part of Theorem \ref{main 0} to conclude that $\clm$ is a $(T_1|_{\clm_{1}},\ldots,T_{n-r}|_{\clm_{1}})$-reducing subspace. That is,
\[
P_\clm T_{j}|_{\clm_{1}}=T_{j}|_{\clm_{1}}P_\clm,
\]
for all $j \in I_{n-r}$. We next claim that $(T_{1}|_{\clm_{1}},\ldots,T_{{n-r}}|_{\clm_{1}})$ is doubly commuting. To see this, first note that since $\clm\subseteq\clm_{1}$, we have
\[
P_\clm=P_\clm P_{\clm_1}=P_{\clm_1}P_\clm.
\]
Fix $1\le i<j\le n-r$. By doubly commutativity of $(T_{1}|_{\clm},\ldots,T_{{n-r}}|_{\clm})$, we have
\[
T_i|_\clm P_\clm T^*_j|_\clm=P_\clm T^*_j|_\clm T_i|_\clm,
\]
and hence
\[
T_i (P_{\clm_1}P_\clm P_{\clm_1}) T^*_j (P_{\clm_1}P_\clm) = (P_\clm P_{\clm_1}) T^*_j (P_{\clm_1}P_\clm P_{\clm_1}) T_i (P_{\clm_1}P_\clm),
\]
which implies
\[
(T_i|_{\clm_1}) P_\clm (P_{\clm_1} T^*_j|_{\clm_1}) P_\clm = P_\clm (P_{\clm_1} T^*_j|_{\clm_1}) P_\clm (T_i|_{\clm_1})P_\clm.
\]
At this point, we use the fact that $\clm$ is $(T_1|_{\clm_{1}},\ldots,T_{n-r}|_{\clm_{1}})$-reducing to obtain
\[
(T_i|_{\clm_1}) (P_{\clm_1} T^*_j|_{\clm_1}) P_\clm = (P_{\clm_1}T^*_j|_{\clm_1}) (T_i|_{\clm_1}) P_\clm.
\]
Finally, since $\clm_1$ reduces $(S_1,\ldots,S_r)$ and the tuple $(T_1,\ldots,T_{n-r},S_1,\ldots,S_r)$ is doubly commuting, applying $\prod\limits_{i=1}^{r}S^{*{m_i}}_i$ from the left to both sides of the above identity  yields
\[
(T_i|_{\clm_1}) (P_{\clm_1} T^*_j|_{\clm_1}) \prod\limits_{i=1}^{r}S^{*{m_i}}_iP_\clm = (P_{\clm_1}T^*_j|_{\clm_1}) (T_i|_{\clm_1}) \prod\limits_{i=1}^{r}S^{*{m_i}}_iP_\clm,
\]
and hence, by the construction of the space $\clm_1$, it follows that $(T_{1}|_{\clm_{1}},\ldots,T_{{n-r}}|_{\clm_{1}})$ is doubly commuting.

\noindent By construction, $\clm_{1} = \clw \otimes \clq_\Theta$ is invariant under $(T_1,\ldots, T_{n-r})$, and hence $\clw$ is a submodule of $H^2(\D^{n-r})$. Since $(T_{1}|_{\clm_1},\ldots,T_{{n-r}}|_{\clm_1})$ is doubly commuting, and $\clm\subseteq \clm_1$ reduces this tuple, Proposition \ref{doubly red} ensure an inner function $\varphi\in H^{\infty}(\D^{n-r})$ and a closed subspace $\cle$ of $\clq_\Theta$ such that
\[
\clm = \varphi H^2(\D^{n-r}) \otimes\cle.
\]
Recall that $S_j = I_{H^2(\D^{n-r})} \otimes S_{\Theta_j}$ for $j \in I_r$. Since $(S_1|_{\clm},\ldots,S_r|_{\clm})$ doubly commute, it follows that $(S_{\Theta_1}|_\cle, \ldots, S_{\Theta_r}|_\cle)$ is also doubly commuting. By Theorem \ref{main 0}, there exist submodules $\clw_{i} \subseteq \clq_{\theta_i}$, $i \in I_r$, such that $\cle=\clw_1\otimes\cdots\otimes\clw_r$. Hence, $\clm=\varphi H^2(\D^{n-r})\otimes\clw_1\otimes\cdots\otimes\clw_r$. The converse follows from the tensor product structure of the operators $T_i$'s and $S_j$'s.
\end{proof}

We now consider a particular case, namely the central result of \cite{MR2380073}, which follows as a special case of the above theorem. This case served as both the underlying motivation and a guiding perspective for the results developed in this paper.

\begin{remark}\label{rem: Nakazi et al}
Let $\theta(z) = z^2$, and consider the mixed space $H^2(\D)\otimes\clq_{\theta}$. The nonzero submodules of $\clq_{\theta}$ are precisely $\clq_{z^2}$ and $z\clq_z$. Let $\clm$ be a submodule of $H^2(\D)\otimes\clq_{\theta}$. By Theorem \ref{main 1}, $\clm$ is doubly commuting if and only if there exists an inner function $\varphi\in H^\infty(\D)$ such that
\[
\clm = \varphi H^2(\D)\otimes\clq_{z^2} \text{ or } \clm = \varphi H^2(\D)\otimes z\clq_{z}.
\]
By identifying $H^2(\D)\otimes\clq_{z^2}$ with $H^2(\mathbb{C}^2)$, this recovers the result proved in \cite{MR2380073}. The present setting is not only much more general, but also provides a new approach to proving the result of \cite{MR2380073}.
\end{remark}

\section{Unitary equivalence}\label{sec: unit equiv}

In this section, we retain the notation introduced in the previous section. It is a celebrated result of Douglas and Foias \cite{MR1243270} that two nontrivial quotient modules of $H^2(\D^n)$ are never unitarily equivalent. In this section, we investigate this problem in the context of doubly commuting submodules of
\[
\clh^2(\Theta) = H^2(\D^{n-r}) \otimes \clq_\Theta,
\]
where $\clq_{\Theta} = \clq_{\theta_1} \otimes \cdots \otimes \clq_{\theta_r}$. Two submodules $\clm_{1}$ and $\clm_{2}$ of $\clh^2(\Theta)$ are unitarily equivalent if there exists a unitary $U:\clm_{1}\to\clm_{2}$ such that
\begin{equation}\label{eqn: U T_i = T_i U}
U(T_i|_{\clm_1}) = (T_i|_{\clm_2})U \text{ and } U(S_i|_{\clm_1})=(S_i|_{\clm_2})U,
\end{equation}
for all $i \in I_{n-r}$, and $j \in I_r$.

\begin{theorem}\label{thm: unit equiv}
Let $\clm_1$ and $\clm_2$ be two doubly commuting submodules of $\clh^2(\Theta)$. Then $\clm_1$ and $\clm_2$ are unitarily equivalent if and only if there exist submodules $\clw_i$ of $\clq_{\theta_i}$ for $i \in I_r$, and inner functions $\varphi,\psi\in H^\infty(\D^{n-r})$ such that
\[
\clm_1=\varphi H^2(\D^{n-r})\otimes \clw_1\otimes\cdots\otimes\clw_r \text{ and } \clm_2=\psi H^2(\D^{n-r})\otimes \clw_1\otimes\cdots\otimes\clw_r.
\]
\end{theorem}
\begin{proof}
Let $\clm_{1}=\varphi H^2(\D^{n-r})\otimes\clw_1\otimes\cdots\otimes\clw_r$ and $\clm_{2}=\psi H^2(\D^{n-r})\otimes\cle_1\otimes\cdots\otimes\cle_r$. Let $U:\clm_{1}\to\clm_{2}$ be a unitary operator satisfying \eqref{eqn: U T_i = T_i U}. It follows that
\begin{equation}\label{eqn: U* and all}
\begin{split}
U^* \prod_{i=1}^{n-r} (I_{\clm_{2}}-T_i P_{\clm_{2}}T^*_i|_{\clm_2}) \prod_{j=1}^{r} (I_{\clm_{2}} & - P_{\clm_{2}}S^*_i S_i|_{\clm_2}) U = \prod_{i=1}^{n-r} (I_{\clm_1}-T_i P_{\clm_1} T^*_i|_{\clm_1})
\\
& \quad \times \prod_{j=1}^{r} (I_{\clm_1} - P_{\clm_1} S^*_i S_i|_{\clm_1}).
\end{split}
\end{equation}
We now recall a general fact: given a doubly commuting submodule $\cls = \eta H^2(\D^{n-r})$, where $\eta \in H^\infty(\D^{n-r})$ is inner, one way to recover $\eta$ is via the following well-known identity (cf. \cite{MR3117337}):
\[
\prod_{i=1}^{n-r} (I_{H^2(\D^{n-r})} - T_{z_i} P_{\cls} T_{z_i}^*)|_\cls = P_{\mathbb{C} \eta}.
\]
Also recall from Lemma \ref{lem: 1 - S*S} that $I - S_{\theta_j}^* S_{\theta_j} = T_z^* {\theta_j} \otimes T_z^* \theta_j$ for all $j \in I_r$. Combining these identities, together with \eqref{eqn: U* and all}, we conclude that
\[
U^*(P_{\mathbb{C}\psi} \otimes G_{1}\otimes\cdots \otimes G_r)U = P_{\mathbb{C} \vp} \otimes F_{1}\otimes\cdots\otimes F_r
\]
where $F_i = P_{\clw_i} (T_z^*\theta_i \otimes T_z^*\theta_i)|_{\clw_i}$ and $G_i = P_{\cle_i} (T_z^*\theta_i \otimes T_z^*\theta_i)|_{\cle_{i}}$ are rank-one operators, $i \in I_r$. Let $f_i$ and $g_i$ be norm-one vectors from $\text{ran} F_i$ and $\text{ran} G_i$, respectively, for $i \in I_r$, such that
\[
U^*(\psi\otimes g_1\otimes\cdots\otimes g_r)=\varphi\otimes f_1\otimes\cdots\otimes f_r.
\]
For each $m\in\mathbb{Z}_+$, we compute
\begin{align*}
U^*(\psi\otimes g_1\otimes\cdots\otimes P_{\cle_{i}}S^{*m}_{\theta_i}g_i\otimes\cdots\otimes g_r) & = U^*P_{\clm_{2}}S_i^{*m}|_{\clm_{2}}(\psi\otimes g_1\otimes\cdots\otimes g_i\otimes\cdots\otimes g_r)
\\
& = P_{\clm_{1}}S_i^{*m}|_{\clm_{1}}U^*(\psi\otimes g_1\otimes\cdots\otimes g_i\otimes\cdots\otimes g_r)
\\
&=P_{\clm_{1}}S_i^{*m}|_{\clm_{1}}(\varphi\otimes f_1\otimes\cdots f_i\otimes\cdots\otimes f_r)
\\
& = \varphi\otimes f_1\otimes\cdots P_{\clw_{i}}S_{\theta_i}^{*m}f_i\otimes\cdots\otimes f_r.
\end{align*}
By Lemma \ref{full}, we know $\bigvee\limits_{m\in\mathbb{Z}_+}P_{\clw_{i}}S^{*m}_{\theta_i} f_i = \clw_{i}$ and $\bigvee\limits_{m \in\mathbb{Z}_+} P_{\cle_{i}}S^{*m}_{\theta_i}g_i = \cle_{i}$. Since $U^*$ is unitary, it follows that the mapping $V_i(P_{\cle_{i}}S^{*m}_{\theta_i}g_i)=P_{\clw_{i}}S^{*m}_{\theta_i}f_i$ defines a unitary operator $V_i:\cle_{i}\to\clw_{i}$ satisfying
\[
V_i(P_{\cle_{i}}S^*_{\theta_i}|_{\cle_{i}})=(P_{\clw_{i}}S^*_{\theta_i}|_{\clw_{i}})V_i,
\]
for all $i \in I_r$. Then Douglas and Foias \cite[Theorem 1]{MR1243270} implies that $\clw_{i}=\cle_{i}$ and $V_i=\lambda_i I_{\clw_{i}}$ for some $\lambda_i\in\mathbb{T}$, for $i \in I_r$. The converse follows easily from the fact that $\varphi H^2(\D^{n-r})$ and $\psi H^2(\D^{n-r})$ are always unitarily equivalent.
\end{proof}

Jordan blocks of $H^2(\D^n)$ are simple yet powerful objects for analyzing the structure of submodules and quotient modules of $H^2(\D^n)$. The theory of invariant subspaces of operators or commuting tuples of operators is itself a subject of independent interest and of wider applicability \cite{DGS, II,Y}. It is therefore natural to expect that a deeper understanding of the invariant subspaces of Jordan blocks of $H^2(\D^n)$ will yield further insight into the general theory of Hilbert spaces of analytic functions. We hope to pursue results along these lines in future work.

\bigskip

\noindent\textbf{Acknowledgement:} This research was initiated during the second author’s visit to the National Defense Academy, Japan, in the April of 2025, and he acknowledges the generous and warm hospitality extended there. Part of the work was also carried out during a workshop at RIMS, Kyoto, in April 2026, and the second and third authors gratefully acknowledge the support of RIMS. The research of the first named author is supported by NBHM (National Board of Higher Mathematics, India) Ph.D. fellowship No. 0203/13(47)/2021-R\&D-II/13177. The research of the second named author is supported in part by MATRICS (ANRF/ARGM/2025/000130/MTR) by ANRF, Department of Science \& Technology (DST), Government of India. The second and third authors were also supported by JSPS KAKENHI Grant Number JP24K06771.

\bibliographystyle{amsplain}

\end{document}